\documentclass[a4paper,11pt,reqno]{amsart}
\usepackage[centering, totalwidth = 390pt, totalheight = 610pt]{geometry}
\usepackage{amsmath,amssymb,amsfonts, url, stmaryrd}
\usepackage{eucal}
\usepackage{verbatim}
\usepackage{amsthm}
\usepackage{xspace}

\usepackage{enumerate}

\numberwithin{equation}{section}

\theoremstyle{plain}

\newtheorem{Theorem}{Theorem}

\newtheorem{Corollary}[Theorem]{Corollary}
\newtheorem{Proposition}[Theorem]{Proposition}
\newtheorem{Lemma}[Theorem]{Lemma}

\theoremstyle{definition}

\newcommand{\Cat}{\ensuremath{\textnormal{Cat}}\xspace}
\newcommand{\Gpd}{\ensuremath{\textnormal{Gpd}}\xspace}

\newcommand{\Set}{\ensuremath{\textnormal{Set}}\xspace}

\newcommand{\f}[1]{\ensuremath{\mathcal{#1}}\xspace}
\newcommand{\g}[1]{\ensuremath{\mathbb{#1}}\xspace}

\newcommand{\Sset}{\ensuremath{\textnormal{SSet}}\xspace}

\newcommand{\TAlg}{\ensuremath{\textnormal{T-Alg}}\xspace}

\newcommand{\TAlgs}{\ensuremath{\textnormal{T-Alg}_{\textnormal{s}}}\xspace}

\newcommand{\dtwo}{\ensuremath{\Delta_{2}}\xspace}

\input xy
\xyoption{all}
\SelectTips{cm}{11}
\makeatletter
\def\matrixobject@{%
 \edef \next@{={\DirectionfromtheDirection@ }}%
 \expandafter \toks@ \next@ \plainxy@
 \let\xy@@ix@=\xyq@@toksix@
 \xyFN@ \OBJECT@}
\let\xy@entry@@norm=\entry@@norm
\def\entry@@norm@patched{%
 \let\object@=\matrixobject@
 \xy@entry@@norm }
\AtBeginDocument{\let\entry@@norm\entry@@norm@patched}
\makeatother

\newcommand{\twocong}[2][0.5]{\ar@{}[#2] \save ?(#1)*{\cong}\restore}
\newcommand{\twoeq}[2][0.5]{\ar@{}[#2] \save ?(#1)*{=}\restore}
\newcommand{\rtwocell}[3][0.5]{\ar@{}[#2] \ar@{=>}?(#1)+/l 0.2cm/;?(#1)+/r 0.2cm/^{#3}}
\newcommand{\ltwocell}[3][0.5]{\ar@{}[#2] \ar@{=>}?(#1)+/r 0.2cm/;?(#1)+/l 0.2cm/^{#3}}
\newcommand{\ltwocello}[3][0.5]{\ar@{}[#2] \ar@{=>}?(#1)+/r 0.2cm/;?(#1)+/l 0.2cm/_{#3}}
\newcommand{\dtwocell}[3][0.5]{\ar@{}[#2] \ar@{=>}?(#1)+/u  0.2cm/;?(#1)+/d 0.2cm/^{#3}}
\newcommand{\dltwocell}[3][0.5]{\ar@{}[#2] \ar@{=>}?(#1)+/ur  0.2cm/;?(#1)+/dl 0.2cm/^{#3}}
\newcommand{\drtwocell}[3][0.5]{\ar@{}[#2] \ar@{=>}?(#1)+/ul  0.2cm/;?(#1)+/dr 0.2cm/^{#3}}
\newcommand{\dthreecell}[3][0.5]{\ar@{}[#2] \ar@3{->}?(#1)+/u  0.2cm/;?(#1)+/d 0.2cm/^{#3}}
\newcommand{\utwocell}[3][0.5]{\ar@{}[#2] \ar@{=>}?(#1)+/d 0.2cm/;?(#1)+/u 0.2cm/_{#3}}
\newcommand{\dtwocelltarg}[3][0.5]{\ar@{}#2 \ar@{=>}?(#1)+/u  0.2cm/;?(#1)+/d 0.2cm/^{#3}}
\newcommand{\utwocelltarg}[3][0.5]{\ar@{}#2 \ar@{=>}?(#1)+/d  0.2cm/;?(#1)+/u 0.2cm/_{#3}}

\newcommand{\atwo}{\textbf{2}\xspace}

\title{A colimit decomposition for homotopy algebras in CAT}
\begin{document}
\author[John Bourke]{John Bourke}
\address{Department of Mathematics and Statistics, Masaryk University, Kotl\'a\v rsk\'a 2, Brno 60000, Czech Republic}
\email{bourkej@math.muni.cz}
\subjclass[2000]{Primary: 18D05, 55P99}

\date{\today}


\thanks{Supported by the Grant agency of the Czech Republic under the grant P201/12/G028.}

\begin{abstract}
Badzioch showed that in the category of simplicial sets each homotopy algebra of a Lawvere theory is weakly equivalent to a strict algebra.  In seeking to extend this result to other contexts Rosick{\'{y}} observed a key point to be that each homotopy colimit in \Sset admits a decomposition into a homotopy sifted colimit of finite coproducts, and asked the author whether a similar decomposition holds in the 2-category of categories \Cat.  Our purpose in the present paper is to show that this is the case.
\end{abstract}
 \leftmargini=2em

\def\xypic{\hbox{\rm\Xy-pic}}
\maketitle

\section{Introduction}
When $\f V$ is a complete and cocomplete symmetric monoidal closed category the theory of categories enriched in $\f V$ develops in much the same way as ordinary category theory.  Classical concepts, such as finite limit theories and their algebras, have enriched analogues: if $\g T$ is a small $\f V$-category with finite products one can consider $\g T$-algebras in $\f V$, which are $\f V$-functors $X:\g T \to \f V$ preserving finite products.

If $\f V$ has a notion of weak equivalence then algebras have a natural homotopy analogue: a \emph{homotopy algebra} being given by a $\f V$-functor $X:\g T \to \f V$ which preserves products up to weak equivalence, in the sense that the canonical map $X(A_{1} \times \ldots \times A_{n}) \to X(A_{1}) \times \ldots \times X(A_{n})$ is a weak equivalence for each finite tuple of objects of $\g T$.  

Each genuine or \emph{strict} algebra is a homotopy algebra and one can ask to what extent the converse is true---with respect to the natural pointwise notion of weak equivalence one can ask whether each homotopy algebra is weakly equivalent to a strict one.  Badzioch in \cite{Badzioch2002Algebraic} investigated this question in the case of simplicially enriched categories, with theories the classical single sorted Lawvere theories viewed as discrete simplicial categories; his main result 
a \emph{rigidification theorem} establishing each homotopy algebra to be weakly equivalent to a strict algebra.  This result was extended by Bergner in \cite{Bergner2005Rigidification} to cover finite product theories, again in the simplicial setting.

In \cite{Rosicky2012Rigidification} Rosick{\'{y}} has investigated the possibility of extending these rigidification results to other settings, by allowing his base of enrichment \f V to be a monoidal model category other than simplicial sets, and by considering weighted limit theories more general than finite product theories.  One of his rigidification results, Theorem 3.3 of \cite{Rosicky2012Rigidification}, requires that each cofibrant weight, or cofibrant object in $[\f J, \f V]$ with its projective model structure, admits a certain kind of colimit decomposition.  He asked the author whether such a colimit decomposition exists in the case that $\f V = \Cat$ with weak equivalences the equivalences of categories, and when the theories under consideration are just finite product theories---the special nature of the colimit decomposition now requiring that each cofibrant object of $[\f J,\Cat]$ can be presented as a sifted colimit of finite coproducts of representables, in which moreover each colimit involved is homotopically well behaved in a manner described in Section 4.2. 

The aim of the present paper is to show, in Theorem 8, that this is the case.  With this result in place Rosick{\'{y}}'s theorems' 3.3 and 5.1 of \cite{Rosicky2012Rigidification} yield rigidification results for homotopy algebras of finite product theories in \Cat---his Theorem 5.1 now asserts that, in \Cat, each homotopy algebra of a finite product theory is weakly equivalent to a strict algebra, a direct analogue of the results in the simplicial setting described above.

Now the cofibrant objects of $[\f J,\Cat]$ are the \emph{flexible} weights of \cite{Bird1989Flexible}.  Flexible limits and colimits have been the subject of much study in 2-category theory, and many of the results required to give the main decomposition, in Theorem~\ref{thm:Decomp2}, are known---our main contribution here is to put these facts together in an appropriate way.  Since the results are spread throughout the literature on 2-category theory, and some without detailed proof, we give a thorough, and reasonably self contained, treatment of all aspects involved in the decomposition, with the intention of making Rosick{\'{y}}'s rigidification result in this 2-categorical setting more easily accessible.

In Section 2 we give the necessary background, beginning with a few brief remarks on weighted limits and colimits.  We recall the notion of a flexible weight, and so flexible limits and colimits, describing the connection with model categories.  Examples of flexible colimits are given and their properties discussed.  We begin the third section by describing those flexible colimits involved in our decomposition of a flexible weight, giving a detailed treatment of reflexive codescent and reflexive lax codescent objects.  We show each of these colimits to be sifted colimits---the case of reflexive codescent objects is in \cite{Lack2002Codescent}.  Combining these results a presentation of each flexible weight as a sifted flexible colimit of coproducts of representables is given---this is the main novel result of the paper.  In the final section we begin by discussing bicolimits and their relationship with ordinary colimits.  We use that filtered colimits are bicolimits in \Cat \cite{Makkai1989Accessible} and reduce from arbitrary coproducts of representables to finite coproducts, giving the final decomposition in Theorem~\ref{thm:Decomp2}.

The author thanks Ji\v{r}\'{i} Rosick{\'{y}} and Stephen Lack for useful discussions on the content of this paper.

\section{Weighted colimits and flexible colimits}
\subsection{Weighted colimits}
When $\f V$ is a complete and cocomplete symmetric \\monoidal closed category one has the full theory of categories enriched in $\f V$ \cite{Kelly1982Basic}.  $\f V$ itself admits such an enrichment (which we will also denote by \f V) and for each small $\f V$-category \f J we have the enriched category $[\f J, \f V]$ whose objects, $\f V$-functors, are called \f J-indexed \emph{weights}.  
Given a diagram $D:\f J \to \f C$ its $W$-\emph{weighted limit} is an object $\{ W,D\}$ together with a $\f V$-natural transformation $W \to \f C(\{W,D\},D-)$ called a \emph{cone}, or cylinder, which induces an isomorphism $$\f C(A, \{W,D\}) \cong [\f J,\f V](W, \f C(A,D-))$$ for each $A \in \f C$.  Dually given a diagram $D:\f J^{op} \to \f C$ its $W$-weighted colimit, or just $W$-colimit to be brief, is an object $W \star D$ equipped with a \emph{cocone} $W \to \f C(D-,W\star D)$ inducing an isomorphism $$\f C(W \star D,A) \cong [\f J,\f V](W, \f C(D-,A))$$ for each $A \in \f C$.  

Amongst weighted colimits we find the familiar conical colimits such as coequalisers and coproducts, and also tensors by objects of $\f V$, but also more complex kinds---for instance when \f V = \Cat we have codescent objects, Kleisli objects of monads and many other useful colimits, some of which are described in detail in Section 3.

\subsection{Flexible colimits and cofibrancy}
A limit $\{W,D\}$ or colimit $W \star D$ is said to be flexible if $W$ is a \emph{flexible weight}.  One can approach flexible weights using model categories without knowing anything of 2-category theory beyond general enriched category theory, and likewise using 2-category theory without any model categories at all.  Both perspectives are important here so we recall each.

Flexible weights were first defined in 2-category theory, as a special case of the notion of flexible algebra for a 2-monad, and the results we describe now are special cases of results of \cite{Blackwell1989Two-dimensional} concerning 2-monads.  Whilst the generality of 2-monads is mostly beyond our needs at present a little background is required.  

Given a 2-category \f C we have the identity on objects inclusion $\iota:[\f J,\f C] \to Ps(\f J,\f C)$ with the latter 2-category having 2-functors as objects, arrows the more general pseudonatural transformations, and modifications for 2-cells.  The inclusion is the identity on objects so that we typically omit to label its action.  If \f C is both complete and cocomplete the inclusion has a left adjoint $Q$---we refer the reader to Section 3.2 for more detail on this.  The unit and counit at $W\in [\f J, \f C]$ are given by pseudonatural and 2-natural transformations $p_{W}:W \rightsquigarrow QW$ and $q_{W}:QW \to W$ respectively.  The isomorphism $$[\f J, \f C](QW,X) \cong Ps(\f J, \f C)(W,X)$$ exhibits $QW$ as a \emph{pseudomorphism classifier}, in the sense that any pseudonatural $W \rightsquigarrow X$ factors uniquely through $p_{W}:W \rightsquigarrow QW$ as a 2-natural transformation.  One of the triangle equations for the adjunction asserts that the pseudonatural $p_{W}:W \rightsquigarrow QW$ is a section of $q_{W}$ in $Ps(\f J, \f C)$; in fact $q_{W}$ is a retract equivalence, or \emph{surjective equivalence}, in the 2-category $Ps(\f J, \f C)$.  It follows in particular that the component of $q_{W}:QW \to W$ at each $j \in \f J$ is a surjective equivalence in \f C; thus $q_{W}$ is a \emph{pointwise surjective equivalence}.

 Now let us interpret the above in the special case of $[\f J, \Cat]$.  We have seen that the counit $q_{W}:QW \to W$ at a weight $W$ always admits a pseudonatural section $p_{W}$; the weight $W$ is said to be a \emph{flexible weight} \cite{Bird1989Flexible} just when $q_{W}$ admits a section in $[\f J,\Cat]$, in which case $q_{W}:QW \to W$ is in fact a surjective equivalence in $[\f J,\Cat]$.

On the other hand if we ignore its 2-dimensional structure then \Cat is a combinatorial model category with weak equivalences and fibrations the equivalences of categories and isofibrations, and trivial fibrations the surjective equivalences.  Cartesian product gives it the structure of a monoidal model category \cite{Hovey1999Model} so that one can speak of model 2-categories, \Cat being one of these.  It was shown in \cite{Lack2007Homotopy-Theoretic} that $[\f J,\Cat]$ obtains the projective model structure in which the weak equivalences and fibrations are pointwise as in \Cat, and that a \emph{flexible weight} is just a cofibrant object, or cofibrant weight, therein.  This was done in the more general context of 2-dimensional monad theory, but since the proof is short we include it in the case of weights.
\begin{Proposition}\textnormal{(Lack)}
When $[\f J,\Cat]$ is equipped with the projective model structure, the cofibrant objects therein are exactly the flexible weights.
\end{Proposition}
\begin{proof}
  In the projective model structure on $[\f J, \Cat]$ the trivial fibrations are pointwise as in \Cat, and so the pointwise surjective equivalences.  We have seen that $q_{W}:QW \to W$ is one of these; as such it will exhibit $QW$ as a cofibrant replacement of $W$ if we can show $QW$ to be cofibrant.  Upon doing so it is clear that $W$ will be cofibrant just when $q_{W}$ admits a section, which is to say when $W$ is a flexible weight.  To see that $QW$ is cofibrant suppose that $f:X \to Y$ is a pointwise surjective equivalence---given an arbitrary arrow $r:QW \to Y$ we should show it factors through $f$.  Observe that as $f$ is a pointwise surjective equivalence it admits a pseudonatural section $g:Y \rightsquigarrow X$ so that $fg=1$.  Now consider the following diagram

$$\xy
(0,0)*+{W}="00";(20,0)*+{QW}="10";(40,0)*+{Y}="20";(40,15)*+{X}="21";
{\ar@{~>}_{p_{W}} "00"; "10"}; 
{\ar_{r} "10"; "20"}; 
{\ar^{f} "21"; "20"}; 
{\ar@{~>}_{gr} "10"; "21"}; 
{\ar@/^1.5pc/^{h} "10"; "21"}; 
\endxy$$
where all but $h$ have been defined.  By the universal property of $p_{W}$ the composite $grp_{W}:W \rightsquigarrow X$ is uniquely of the form $hp_{W}$ for a 2-natural $h:QW \to X$.  Now to check that $fh=r$ it suffices, by the same universal property, to show $fhp_{W}=rp_{W}$  But we have $fhp_{W}=fgrp_{W}=rp_{W}$ as required.
\end{proof}

Knowing that the flexible weights are the cofibrant weights doesn't offer much insight as to what they actually look like, or why they are interesting in 2-category theory.  Let us conclude this section by briefly mentioning some examples of flexible colimits including the generating ones, and some other properties.  We will describe the examples of most importance in more detail in Section 3.

\subsection{Pseudocolimits}
Given a weight $W$ and diagram $D: \f J^{op} \to \f C$ its $W$ weighted \emph{pseudocolimit}  $W \star _{p} D$ is defined by an isomorphism $\f C(W \star_{p} D, A) \cong Ps(\f J,\Cat)(W, \f C(D-,A))$ natural in $A$.  By the adjunction $Q \dashv \iota$ we have a natural isomorphism $[\f J,\Cat](QW, \f C(D-,A)) \cong Ps(\f J,\Cat)(W, \f C(D-,A))$ so that the pseudocolimit is nothing but the weighted colimit $QW \star D$.  That pseudocolimits are flexible is easy to see: the adjunction $Q \dashv \iota$ generates a comonad $(Q,q,\Delta)$ on $[\f J,\Cat]$ with counit the same $q:Q \to 1$ as before; in particular each $QW$ admits a (co-free) coalgebra structure, and so is certainly a flexible weight.  Since the pseudocolimit $W \star_{p} D$ is the genuine colimit of $D$ weighted by a cofibrant replacement of $W$, pseudocolimits are closely related to homotopy colimits---this relationship was studied in \cite{Gambino2008Homotopy}.


\subsection{Saturation, pie colimits and splittings of idempotents}
Let $Flex$ denote the class of flexible weights so that $Flex(\f J) \subset [\f J, \Cat]$ consists of the \f J-indexed flexible weights.  It was shown in \cite{Bird1989Flexible} that the flexible weights form a saturated class\begin{footnote}{Saturated classes were originally called closed in \cite{Albert1988The-closure}.}\end{footnote} in the sense of \cite{Kelly2005Notes}.  This means that, for each \f J, $Flex(\f J)$ contains the representables and is closed in $[\f J,\Cat]$ under flexible colimits (not just \f J-indexed ones).  Moreover four kinds of colimit suffice to construct all flexible weights---namely co(p)roducts, co(i)nserters, co(e)quifiers and (s)plittings of idempotents, by which we mean that these four colimits are flexible and $Flex(\f J)=PIES^{*}(\f J)$, the closure of the representables in $[\f J,\Cat]$ under these four kinds of colimit.  

It follows that if one can construct $W$-colimits out of these four kinds in a general 2-category then $W$ is flexible: one uses that $W=W \star Y$ for the Yoneda embedding $Y$ and carries out the corresponding construction in $[\f J,\Cat]$.  This fact provides a convenient way to test whether a particular weight is flexible---examples which are easily seen to be flexible in this manner are coinverters, Kleisli objects of monads, the codescent objects of the following section along with numerous others---many such cases were described in \cite{Kelly1989Elementary} and \cite{Bird1989Flexible}.

Though not relevant in what follows it is perhaps worth mentioning that if we drop splittings of idempotents from the above and take the closure $PIE^{*}(\f J) \subset [\f J,\Cat]$ we get what are called the \emph{pie weights} \cite{Power1991A-characterization}.  Apart from splittings of idempotents essentially all flexible colimits, such as those just mentioned, are pie.  The pie weights can be recognised as precisely those admitting coalgebra structure for the comonad $Q$ on $[\f J,\Cat]$---see \cite{Lack2011Enhanced} or \cite{Bourke2011On-semiflexible}---this can be interpreted as saying that they are the \emph{algebraically cofibrant} objects in $[\f J,\Cat]$ in the sense of \cite{Riehl2011Algebraic}, a perspective which was further explored in \cite{Bourke2011On-semiflexible}.
\subsection{The importance of flexible limits and colimits in 2-category theory}

Let us briefly indicate some reasons for the interest in flexibility.  Primary objects of study in 2-dimensional universal algebra are 2-categories, such as the 2-category of monoidal categories and strong monoidal functors, whose morphisms only preserve structure up to isomorphism.  Such 2-categories generally admit pie limits \cite{Blackwell1989Two-dimensional}, and if the structure involved is itself ``flexible", such as monoidal structure, they also admit flexible limits \cite{Bird1989Flexible}; note that in the full sub 2-category containing the more ``rigid" strict monoidal categories idempotents need not split.  The precise distinction is that monoidal categories are the algebras for a \emph{flexible 2-monad} \cite{Kelly1974Doctrinal, Blackwell1989Two-dimensional} whereas strict monoidal categories are not.

\section{Sifted flexible colimits and a first decomposition}
A weight $W \in [\f J,\Cat]$ is said to be \emph{sifted} if finite products commute with $W$-colimits in \Cat.  This is to say that the 2-functor $W \star -:[\f J^{op},\Cat] \to \Cat$ preserves finite products.  In the present section we describe a number of weights which are both sifted and flexible and give our first decomposition result.  

The two key kinds of colimits are reflexive codescent and reflexive lax codescent objects---that the former are sifted is a result of Lack, Proposition 4.3 of \cite{Lack2002Codescent}, though the proof only outlined.  We fill in the details here and follow a suggestion of Lack to extend this result to the lax setting.  In both cases we follow the argument outlined in \cite{Lack2002Codescent}---to apply the following lemma\begin{footnote}{
In \cite{Lack2002Codescent} the hypothesis on the terminal object does not appear in the statement of the lemma but is discussed in the proof.}\end{footnote} of the same paper, which reduces the colimits to be computed to manageable special cases.
\begin{Lemma}\label{lem:Lack}\textnormal{(Lack)}
A weight $W:\f J \to \Cat$ is sifted if $W \star -: [\f J^{op},\Cat] \to \Cat$ preserves finite products of representables.  If \f J has a terminal object which is preserved by $W$ then $W$ is sifted so long as $W \star -$ preserves binary products of representables.
\end{Lemma}

Let us remark upon our terminology concerning codescent objects---this is based upon \cite{Lack2002Codescent} and fits well with the appearance of codescent objects in 2-dimensional monad theory.  What we call lax codescent and codescent objects have also been called codescent and isocodescent objects respectively---see \cite{Street2004Categorical-and} for instance.  We will only consider the notion of a \emph{reflexive} lax codescent or \emph{reflexive} codescent object here which relate to the irreflexive kind \cite{Lack2002Codescent} as reflexive coequalisers do to coequalisers---a notable distinction is that only the reflexive variants commute with finite products in \Cat and \Set respectively.

\subsection{Reflexive lax codescent objects}
Truncating the simplicial category $\Delta$ at the ordered set with three elements gives a full subcategory $\dtwo \subset \Delta$; now restricting the usual embedding $\Delta \to \Cat$ along the inclusion yields the weight $W_{l}:\dtwo \to \Cat$ for reflexive lax codescent objects.  A diagram $\dtwo^{op} \to \f C$ in a 2-category \f C consists of a truncated simplicial object as on the left below
$$\xy
(-10,0)*+{(1)}="x0";
(0,0)*+{A_{2}}="c0"; (20,0)*+{A_{1}}="b0";(40,0)*+{A_{0}}="a0";
{\ar@<1.5ex>^{d} "b0"; "a0"}; 
{\ar@<0ex>|{i} "a0"; "b0"}; 
{\ar@<-1.5ex>_{c} "b0"; "a0"}; 
{\ar@<3ex>^{p} "c0"; "b0"}; 
{\ar@<0ex>|{m} "c0"; "b0"}; 
{\ar@<-3ex>_{q} "c0"; "b0"};
{\ar@<1.5ex>|{r} "b0"; "c0"};
{\ar@<-1.5ex>|{l} "b0"; "c0"};
\endxy
\hspace{1cm}
\xy
(15,0)*+{A_{1}}="b0"; (25,8)*+{A_{0}}="c0"; (35,0)*+{A}="d0"; (25,-8)*+{A_{0}}="e0";
{\ar ^{d} "b0";"c0"};
{\ar _{c} "b0";"e0"};
{\ar ^{f}"c0";"d0"};
{\ar_{f} "e0";"d0"};
{\ar@{=>}^{\eta}(25,3)*+{};(25,-3)*+{}};
\endxy$$
In elementary terms its reflexive lax codescent object $A$ is specified by a triple ($A$, $f:A_{0} \to A$, $\eta:fd \Rightarrow fc)$ as above satisfying the two equations for a \emph{lax codescent cocone}.  The first of these asserts the equality
$$\xy
(0,0)*+{A_{2}}="a0";(15,0)*+{A_{1}}="b0"; (25,12)*+{A_{0}}="c0"; (35,0)*+{A}="d0"; (25,-12)*+{A_{0}}="e0"; (10,12)*+{A_{1}}="f0";  (10,-12)*+{A_{1}}="g0";
{\ar ^{p} "a0";"f0"};
{\ar _{q} "a0";"g0"};
{\ar ^{d} "f0";"c0"};
{\ar _{c} "g0";"e0"};
{\ar ^{d} "b0";"c0"};
{\ar _{c} "b0";"e0"};
{\ar ^{f}"c0";"d0"};
{\ar_{f} "e0";"d0"};
{\ar^{m} "a0" ;"b0"};
{\ar@{=>}^{\eta}(25,3)*+{};(25,-3)*+{}};
\endxy
\hspace{0.5cm}
\xy
(0,0)*+{=};
\endxy
\hspace{0.5cm}
\xy
(0,0)*+{A_{2}}="a0";(20,0)*+{A_{0}}="b0"; (25,12)*+{A_{0}}="c0"; (35,0)*+{A}="d0"; (25,-12)*+{A_{0}}="e0"; (10,12)*+{A_{1}}="f0";  (10,-12)*+{A_{1}}="g0";
{\ar ^{p} "a0";"f0"};
{\ar _{q} "a0";"g0"};
{\ar ^{d} "f0";"c0"};
{\ar _{c} "g0";"e0"};
{\ar _{c} "f0";"b0"};
{\ar ^{d} "g0";"b0"};
{\ar |{f} "b0";"d0"};
{\ar ^{f}"c0";"d0"};
{\ar_{f} "e0";"d0"};
{\ar@{=>}^{\eta}(24,8)*+{};(24,3)*+{}};
{\ar@{=>}^{\eta}(24,-3)*+{};(24,-8)*+{}};
\endxy$$
whilst the second equation asserts that $\eta i: f=fdi \Rightarrow fci=f$ is an identity 2-cell.  As with all 2-categorical colimits it has both a 1 and 2-dimensional aspect to its universal property; the 1-dimensional aspect asserts that given any other such cocone $(B,g,\theta)$ there exists a unique arrow $g^{\prime}:A \to B$ such that $g^{\prime}f=g$ and $g^{\prime}\eta = \theta$; the 2-dimensional aspect asserts that given a second such triple $(B,h,\phi)$ together with a 2-cell $\rho:g \Rightarrow h$ rendering the square
$$
\xy
(0,0)*+{gd}="a0"; (15,0)*+{gc}="b0";(0,-10)*+{hd}="c0";(15,-10)*+{hc}="d0";
{\ar@{=>}^{\theta} "a0"; "b0"}; 
{\ar@{=>}_{\rho d} "a0"; "c0"}; 
{\ar@{=>}^{\rho c} "b0"; "d0"}; 
{\ar@{=>}_{\phi} "c0"; "d0"}; 
\endxy$$
commutative, then there exists a unique 2-cell $\rho^{\prime}:g^{\prime} \Rightarrow h^{\prime}$ between the induced factorisations such that $\rho^{\prime} \circ f = \rho$.

The most important example of such colimits for our concerns is the following example from \cite{Street2004Categorical-and}.  A small category $A$ can be presented as an internal category in \Set by taking its truncated nerve.  The reader can interpret diagram (1) above in this manner, so that $A_{0}, A_{1}$ and $A_{2}$ are the sets of objects, arrows and composable pairs of $A$, with the maps $d$ and $c$ the domain and codomain projections, and so on.  Viewing each of these sets as discrete categories one can view this internal category in \Set as an internal category in \Cat whose reflexive lax codescent object is exactly $A$.  The universal cocone $(A,f,\eta)$ has $f:A_{0} \to A$ the identity on objects inclusion, and $\eta:fd \Rightarrow fc$ the natural transformation whose component $\eta_{\alpha}$ at an object $\alpha:x \to y \in A_{1}$ is simply  $\alpha:x = fd\alpha \to fc \alpha = y$ itself, now viewed as an arrow of $A$.  We leave to the reader the worthwhile exercise of checking that this is indeed the claimed colimit.

In order to see that reflexive lax codescent objects are sifted colimits in \Cat it will be worth being precise about the manner in which we passed from the category $A$ to the corresponding internal category in \Cat; this was achieved by taking the singular \emph{functor} $\Cat \to [\dtwo^{op},\Set]$ associated to the embedding $\dtwo \to \Cat$ and postcomposing by $D_{*}:[\dtwo^{op},\Set] \to [\dtwo^{op},\Cat]$ where $D:\Set \to \Cat$ is the functor viewing each set as a discrete category; let us write $N_{d}:\Cat \to [\dtwo^{op},\Cat]$ for the composite functor (note that this is \emph{not} a 2-functor).  As a composite of limit preserving functors we observe that \emph{$N_{d}$ preserves limits}.

Consider the 2-functor $W_{l} \star - :[\dtwo^{op},\Cat] \to \Cat$ which takes reflexive lax codescent objects and write $(W_{l} \star -)_{0}: [\dtwo^{op},\Cat] \to \Cat$ for its underlying functor.  The content of our example is that we have an isomorphism $W_{l} \star N_{d}(A) \cong A$ for each category $A$; moreover this is easily seen to be a natural isomorphism $(W_{l} \star -)_{0} \circ N_{d} \cong 1$.

A diagram $X \in [\dtwo^{op},\Cat]$ lies in the essential image of $N_{d}$ just when it is an internal category in \Cat with each $X(j)$ discrete; let us call such diagrams in \Cat \emph{pointwise discrete categories}.  The relevance of this notion is that the representables $\dtwo(-,i):\dtwo^{op} \to \Cat$, corresponding to our weight $W_{l}:\dtwo \to \Cat$, \emph{all share this form}.  To see this let  $i \in \{0,1,2\}$ and consider the following composite
$$
\xy
(0,0)*+{\dtwo^{op}}="a0"; (20,0)*+{\Delta^{op}}="b0";(40,0)*+{\Set}="c0";(60,0)*+{\Cat}="d0";
{\ar^{j} "a0"; "b0"}; 
{\ar^{\Delta(-,i)} "b0"; "c0"}; 
{\ar^{D} "c0"; "d0"}; 
\endxy$$
In the middle the representable $\Delta(-,i)$ is well known to be the nerve of a category, so that its restriction $\Delta(-,i) \circ j$ is one too.  As $\dtwo$ is a full subcategory of $\Delta$ so the restriction $\Delta(-,i) \circ j$ is just $\dtwo(-,i)$ which is hence an internal category.  But we are supposed to be considering $\dtwo$ as a 2-category and its \Cat-valued representables; however since $\dtwo$ is locally discrete the corresponding \Cat-valued representable is just the composite $D \circ \dtwo(-,i)$, which, as $D$ preserves pullbacks, is a pointwise discrete category in \Cat.  We can now prove:

\begin{Proposition}\label{prop:cod}
Reflexive lax codescent objects are sifted flexible colimits.
\end{Proposition}
\begin{proof}
Reflexive lax codescent objects can be constructed by forming a coinserter followed by two coequifiers and are consequently flexible colimits (see 2.4)---this construction is described in Proposition 2.1 of \cite{Lack2002Codescent} for a more general kind of lax codescent object. 

With regards siftedness observe that \dtwo has a terminal object preserved by its inclusion to \Cat so that it suffices, by Lemma~\ref{lem:Lack}, to show that $W_{l} \star - :[\dtwo^{op},\Cat] \to \Cat$ preserves binary products of representables.  Each representable is a pointwise discrete category---we will show $W_{l} \star -$ preserves binary products of these instead.  To show $W_{l} \star -$ preserves the product $X \times Y$ of such a pair it suffices to show that its underlying functor $(W_{l} \star -)_{0}$ does so.  But now $X$ and $Y$ lie in the essential image of $N_{d}$ so that, as \Cat has products and $N_{d}$ preserves them, the product also lies in the essential image.  Consequently we need only show that the composite $(W_{l} \star -)_{0} \circ N_{d}$ preserves binary products.  Being naturally isomorphic to the identity functor this is the case.
\end{proof}

\subsection{Reflexive codescent objects}
The weight $W_{i}:\dtwo \to \Cat$ for reflexive codescent objects is obtained from the weight $W_{l}:\dtwo \to \Cat$ for reflexive lax codescent objects by postcomposing $W_{l}$ by the reflection $\Cat \to \Gpd$ to groupoids, and then passing back via the inclusion $\Gpd \to \Cat$.  In elementary terms, given a diagram as in (1), its reflexive codescent object is specified by a triple $(A,f,\eta)$ satisfying the same equations as in the lax case, with the exception that $\eta$ is now required to be an invertible 2-cell; moreover the 1-dimensional universal property of $A$ only quantifies over triples $(B,g,\theta)$ in which $\theta$ is invertible, whilst the 2-dimensional universal property is the same as before.

The relevant example concerns the construction of the pseudomorphism classifier $QX$ of a diagram $X:\f J \to \f C$ in a complete and cocomplete 2-category \f C as arises from the adjunction $(\iota:[\f J,\f C] \leftrightarrows Ps(\f J,\f C):Q)$ discussed in 2.2.  To explain how this goes observe that restriction $U:[\f J, \f C] \to [ob \f J, \f C]$ along the inclusion of the discrete 2-category with the same objects as \f C has a left 2-adjoint $F$ and that $U$ is moreover monadic.  The adjunction $(\epsilon,F \dashv U,\eta)$ gives $FU$ the structure of a comonad on $[\f J, \f C]$ and so, in the usual way, yields for each $X \in [\f J,\f C]$ a (truncated) simplicial object
\vspace{0.3cm}
$$
\xy
(0,0)*+{(FU)^{3}X}="c0"; (30,0)*+{(FU)^{2}X}="b0";(60,0)*+{FUX}="a0";
{\ar@<1.5ex>^{\epsilon_{FUX}} "b0"; "a0"}; 
{\ar@<0ex>|{F\eta_{UX}} "a0"; "b0"}; 
{\ar@<-1.5ex>_{FU\epsilon_{X}} "b0"; "a0"}; 
{\ar@<3ex>^{} "c0"; "b0"}; 
{\ar@<0ex>|{} "c0"; "b0"}; 
{\ar@<-3ex>_{} "c0"; "b0"};
{\ar@<1.5ex>|{} "b0"; "c0"};
{\ar@<-1.5ex>|{} "b0"; "c0"};
\endxy
$$
\vspace{0.1cm}

 where we have omitted to label the higher face and degeneracy maps.  The reflexive codescent object of this diagram in $[\f J, \f C]$ is exactly $QX$.  

That this is the case is best understood in terms of two dimensional monad theory: the monadic adjunction $F \dashv U$ induces a 2-monad $T=UF$ whose 2-category of strict algebras and strict morphisms \TAlgs is $[\f J, \f C]$ whilst the $T$-pseudomorphisms, as belonging to the 2-category \TAlg, are precisely the pseudonatural transformations of $Ps(\f J, \f C)$---this is shown in Section 6.6 of \cite{Blackwell1989Two-dimensional}.  It follows that the inclusion $[\f J, \f C] \to Ps(\f J, \f C)$ coincides with the inclusion $\TAlgs \to \TAlg$ that views each strict $T$-algebra morphism as a pseudomorphism.  The formula for $QD$ above is then a special case of the formula for the pseudomorphism classifier $QA$ of a $T$-algebra $A$ as a reflexive codescent object of free algebras.  This formula was first described in \cite{Lack2002Codescent} in a more general setting relevant to pseudoalgebras---a description better suited to the present level of generality is given in  Section 4.2 of \cite{Lack2010A-2-categories}.  The most important case for us is when $\f C=\Cat$---at a weight $W$ the above presentation exhibits $QW$ as a reflexive codescent object of \emph{free} weights.

In understanding that reflexive codescent objects are sifted colimits in \Cat it will be useful to break down their construction into two steps: as a reflexive lax codescent object followed by a coinverter.  It is not worth the effort to describe the weight for coinverters here---see \cite{Kelly1989Elementary}---it suffices to say that the domain 2-category consists of a single 2-cell, so that, correspondingly, one forms the coinverter of a 2-cell $\alpha \in \f C(A,B)(f,g)$ in a 2-category \f C; this consisting of a pair $(C,h)$ as on the left below
$$\xy
(0,0)*+{A}="b0"; (15,0)*+{B}="c0"; (30,0)*+{C}="d0";
{\ar@/^1.3pc/^{f} "b0";"c0"};
{\ar@/_1.3pc/_{g}"b0";"c0"};
{\ar@{=>}^{\alpha}(7,3)*+{};(7,-3)*+{}};
{\ar^{h} "c0";"d0"};
\endxy
\hspace{2cm}
\xy
(15,0)*+{A_{1}}="b0"; (25,8)*+{A_{0}}="c0"; (35,0)*+{A}="d0"; (25,-8)*+{A_{0}}="e0";(50,0)*+{B}="f0"; 
{\ar ^{d} "b0";"c0"};
{\ar _{c} "b0";"e0"};
{\ar ^{f}"c0";"d0"};
{\ar_{f} "e0";"d0"};
{\ar@{=>}^{\eta}(25,3)*+{};(25,-3)*+{}};
{\ar_{f} "e0";"d0"};
{\ar^{g} "d0";"f0"};
\endxy
$$
in which $h \alpha$ is invertible; the 1-dimensional universal property is that given any $k:B \to D$ with $k\alpha$ invertible there exists a unique $k^{\prime}:C \to D$ such that $k^{\prime}h=k$; its 2-dimensional universal property asserts that, for each object $D$, the induced functor $\f C(h,D):\f C(C,D) \to \f C(B,D)$ is fully faithful.

Now it is easily seen that given a diagram as in (1) we can form its reflexive codescent object in two steps---firstly forming the reflexive lax codescent object $(A,f,\eta)$ and then the coinverter $(B,g)$ of the resulting 2-cell $\eta$, as drawn on the right above.  In the following proposition we shall use this construction to show that reflexive codescent objects are sifted colimits in \Cat; however it is not the case that coinverters are themselves sifted, but only \emph{reflexive} coinverters---this was shown using a $3 \times 3$ argument in \cite{Kelly1993Coinverters}.  The distinction between a reflexive coinverter and a coinverter is in the input; a coinverter is said to be reflexive when the input 2-cell $\alpha$, as above, admits a \emph{splitting}, in the sense of an arrow $k:B \to A$ such that $fk=1,gk=1$ and $\alpha k=1$.  

We need one further auxiliary concept---that of a \emph{liberal} arrow.  An arrow $f:A \to B$ of a 2-category \f C is said to be liberal if it is conservative in $\f C^{op}$---this means that a 2-cell $\alpha:g \Rightarrow h \in \f C(B,C)$ out of $B$ is invertible whenever the composite $\alpha f$ is.  We will have use for the fact that \emph{each bijective on objects functor is liberal in} \Cat.  Finally observe that given a diagram
$$\xy
(-15,0)*+{D}="a0";
(0,0)*+{A}="b0"; (15,0)*+{B}="c0"; (30,0)*+{C}="d0";
{\ar@/^1.3pc/^{f} "b0";"c0"};
{\ar@/_1.3pc/_{g}"b0";"c0"};
{\ar@{=>}^{\alpha}(7,3)*+{};(7,-3)*+{}};
{\ar^{h} "c0";"d0"};
{\ar^{e}"a0";"b0"};
\endxy$$
with $e$ liberal then $h$ exhibits $C$ as the coinverter of $\alpha$ if and only if it exhibits $C$ as the coinverter of $\alpha e$; this follows from the fact that for any $r:B \to E$ the composite $r\alpha$ is invertible just when $r\alpha e$ is.  With this in place we can prove:

\begin{Proposition}\label{prop:isocod}\textnormal{(Lack)}
Reflexive codescent objects are sifted flexible colimits.
\end{Proposition}
\begin{proof}
That reflexive codescent objects are flexible follows from their construction via reflexive lax codescent objects and coinverters, both of which are flexible colimits.  That coinverters are flexible, constructible from coinserters and coequifiers, can be found in Proposition 4.2 of \cite{Kelly1989Elementary}.

With regards siftedness observe that the weights $W_{i}$ and $W_{l}$ have the same domain so that the associated representables coincide---these are pointwise discrete categories as in 3.1.  Again $W_{i}$ preserves the terminal object so that, as in the proof of Propostion~\ref{prop:cod}, $W_{i}$ will be sifted if we can show that the composite $(W_{i} \star -)_{0} \circ N_{d}:\Cat \to [\dtwo^{op},\Cat] \to \Cat$ preserves binary products.  We will show this to be true by breaking this functor down into several components.  Given $A \in \Cat$ the reflexive codescent object of $N_{d}A$ is obtained by forming the reflexive lax codescent object of $N_{d}A$---this is just $(A,f,\eta)$ as described in 3.1---followed by the coinverter $(B,g)$ of $\eta$ as on the left below
$$\xy
(15,0)*+{A_{1}}="b0"; (25,8)*+{A_{0}}="c0"; (35,0)*+{A}="d0"; (25,-8)*+{A_{0}}="e0";(50,0)*+{B}="f0"; 
{\ar ^{d} "b0";"c0"};
{\ar _{c} "b0";"e0"};
{\ar ^{f}"c0";"d0"};
{\ar_{f} "e0";"d0"};
{\ar@{=>}^{\eta}(25,3)*+{};(25,-3)*+{}};
{\ar_{f} "e0";"d0"};
{\ar^{g} "d0";"f0"};
\endxy
\hspace{1cm}
\xy
(5,0)*+{A_{1}}="b0"; (20,0)*+{A^{\atwo}}="e0"; (40,0)*+{A}="d0";(55,0)*+{B}="f0"; 
{\ar ^{t}"b0";"e0"};
{\ar@{=>}^{\lambda}(30,3)*+{};(30,-3)*+{}};
{\ar@/^1.5pc/^{p} "e0";"d0"};
{\ar@/_1.5pc/_{q}"e0";"d0"};
{\ar^{g} "d0";"f0"};
\endxy$$
Let us form the arrow category $A^{\atwo}$ of $A$ which comes equipped with an evident pair of projections and natural transformation $\lambda:p \Rightarrow q \in \Cat(A^{\atwo},A)$ with the universal property that any natural transformation into $A$ factors uniquely through it; we factorise $\eta$ as $\lambda t$ accordingly as indicated on the right above.  Explicitly $t$ is given by the map which assigns to an object of $A_{1}$, an arrow of $A$, the corresponding object of the arrow category $A^{\atwo}$; thus $t$ is bijective on objects.  As such it is liberal so that the coinverter $B$ of $\eta$ is equally the coinverter of $\lambda$.  Therefore $(W_{i} \star -)_{0} \circ N_{d}:\Cat \to [\dtwo^{op},\Cat] \to \Cat$ is equally just the functor which first assigns to a category $A$ the 2-cell $(A^{\atwo},\lambda:p \Rightarrow q,A)$ and then its coinverter; certainly the first assignment preserves products since arrow categories are limits (cotensors with the free arrow) in \Cat.  Furthermore the 2-cell $\lambda$ is reflexive, split by the functor $i:A \to A^{\atwo}$ that assigns to an object of $A$ the identity arrow upon it; since reflexive coinverters commute with finite products we deduce the claim.
\end{proof}

\subsection{Decomposition of a flexible weight as a sifted flexible colimit of coproducts of representables}
As well as using the above kinds of codescent object in the following decomposition, we will also split an idempotent.  Splittings of idempotents are of course absolute colimits and so sifted; and flexible as discussed in 2.4.  With these results in place we give our first decomposition.

\begin{Theorem}\label{thm:Decomp1}
Each flexible weight lies in the closure of the coproducts of representables under sifted flexible colimits.\end{Theorem}
\begin{proof}
If $W$ is a flexible weight then $q_{W}:QW \to W$ has a section $w:W \to QW$ so that $W$ is the splitting of the idempotent $w \circ q_{W}$ on $QW$, a sifted flexible colimit.  Now recall from 3.2 the adjunction ($U:[\f J, \Cat] \leftrightarrows [ob \f J, \Cat]:F$) given by restriction and left Kan extension along the inclusion $ob \f J \to \f J$.  As described in 3.2 the weight $QW$ is a reflexive codescent object of free weights, those of the form $FX$ for a family $X:ob \f J \to \Cat$, and so a sifted flexible colimit of free weights by Proposition~\ref{prop:isocod}.  For $j \in \f J$ each category $Xj$ is, by 3.1, the reflexive lax codescent object of its truncated nerve $N_{d}(Xj)$ as (unlabelled) on the left below
\vspace{0.2cm}
$$\xy
(0,0)*+{X(j)_{2}}="c0"; (20,0)*+{X(j)_{1}}="b0";(40,0)*+{X(j)_{0}}="a0";
{\ar@<1.5ex>^{} "b0"; "a0"}; 
{\ar@<0ex>|{} "a0"; "b0"}; 
{\ar@<-1.5ex>_{} "b0"; "a0"}; 
{\ar@<3ex>^{} "c0"; "b0"}; 
{\ar@<0ex>|{} "c0"; "b0"}; 
{\ar@<-3ex>_{} "c0"; "b0"};
{\ar@<1.5ex>|{} "b0"; "c0"};
{\ar@<-1.5ex>|{} "b0"; "c0"};
\endxy
\hspace{2cm}
\xy
(0,0)*+{X_{2}}="c0"; (20,0)*+{X_{1}}="b0";(40,0)*+{X_{0}}="a0";
{\ar@<1.5ex>^{} "b0"; "a0"}; 
{\ar@<0ex>|{} "a0"; "b0"}; 
{\ar@<-1.5ex>_{} "b0"; "a0"}; 
{\ar@<3ex>^{} "c0"; "b0"}; 
{\ar@<0ex>|{} "c0"; "b0"}; 
{\ar@<-3ex>_{} "c0"; "b0"};
{\ar@<1.5ex>|{} "b0"; "c0"};
{\ar@<-1.5ex>|{} "b0"; "c0"};
\endxy
\vspace{0.2cm}$$
so that $X$ is the reflexive lax codescent object in $[ob \f J,\Cat]$ of the diagram on the right which pointwise evaluates to that on the left.  Applying the left adjoint $F:[ob \f J, \Cat] \to [\f J,\Cat]$ we deduce that $FX$ is a reflexive lax codescent object of the $FX_{i}$, in particular a sifted flexible colimit of the $FX_{i}$ by Proposition~\ref{prop:cod}.  

For each $i \in \{0,1,2\}$ the family $X_{i}:ob \f J \to \Cat$ takes its values amongst the discrete categories.  Any $Y:ob \f J \to \Cat$ with this property is the coproduct $Y= \Sigma_{j}Y(j).ob \f J(j,-)$ where, since $ob \f J$ is discrete, the representable $ob \f J(j,-)$ is just the characteristic function at $j$.  By the left Kan extension formula for $F$ it is easy to see that $F( ob \f J(j,-)) = \f J(j-)$, whence $FY=F (\Sigma_{j}Y(j).ob \f J(j,-)) = \Sigma_{j}Y(j).\f J(j,-)$ is a coproduct of representables; in particular for each $i \in \{0,1,2\}$ the weight $F(X_{i})$ is a coproduct of representables.  Thus $W$ can be formed in three steps by taking sifted flexible colimits of coproducts of representables.

\end{proof}

\section{Filtered colimits and the reduction to finite coproducts}
In this final section we extend the decomposition of Theorem~\ref{thm:Decomp1}, using filtered colimits to reduce from arbitrary coproducts of representables to finite coproducts.  Whilst filtered colimits are not flexible in general they do exhibit some good homotopical behaviour in \Cat which is crucial for our decomposition---namely, they are bicolimits in \Cat.
\subsection{Bicolimits}
Given a weight $W:\f J \to \Cat$ and diagram $D:\f J^{op} \to \f C$ the $W$-bicolimit $W \star_{b} D$ \cite{Kelly1989Elementary} is specified by a \emph{pseudococone} $W \rightsquigarrow \f C(D-,W \star_{b} D) \in Ps(\f J,\Cat)$ such that the induced functor $$\f C(W \star_{b} D, A) \to Ps(\f J,\Cat)(W, \f C(D-,A))$$ is an equivalence for each $A \in \f C$.  The bicolimit is only determined up to equivalence by its defining property; the pseudocolimit, if it exists, provides a canonical instance.   

We call a genuine colimit $W \star D$ a bicolimit if its colimiting cocone $W \to \f C(D-, W \star D)$ in fact exhibits $W \star D$ as the $W$-bicolimit.  This amounts to saying that the composite 
$$\f C(W \star D, A) \cong [\f J,\Cat](W, \f C(D-,A)) \to Ps(\f J,\Cat)(W, \f C(D-,A))$$
is an equivalence for each $A \in \f C$.  Since the first component is an isomorphism this is equally to say that the inclusion $$[\f J,\Cat](W, \f C(D-,A)) \to Ps(\f J,\Cat)(W, \f C(D-,A))$$ is an equivalence for each $A \in \f C$.

If $W$ is flexible then $q_{W}:QW \to W$ is an equivalence in $[\f J,\Cat]$ so that $q_{W}^{*}:[\f J,\Cat](W, \f C(D-,A)) \to [\f J,\Cat](QW, \f C(D-,A))$ is an equivalence for each $A$.  Composing this map with the canonical isomorphism $[\f J,\Cat](QW, \f C(D-,A))\cong Ps(\f J,\Cat)(W,\f C(D-,A))$ yields the above inclusion, so that \emph{any flexible colimit is a bicolimit}.

It is not true that filtered colimits are flexible nor that they are bicolimits in each 2-category.  They are, however, bicolimits in $\Cat$---this was shown in Lemma 5.4.9 of \cite{Makkai1989Accessible} by directly calculating filtered colimits in \Cat.  We give a different proof below, which follows easily from the equivalence of (1) and (3) in the following.

\begin{Proposition}\label{prop:bicolimits}
Let \f C be a complete and cocomplete 2-category consider a weight $W \in [\f J,\Cat]$.  The following are equivalent.
\begin{enumerate}
\item  $W$-colimits are bicolimits in \f C.
\item For each diagram $D$ the map $q_{W} \star D : QW \star D \to W \star D$ is an equivalence in \f C.
\item For each pointwise equivalence $f:D \to E$ of diagrams the induced $W \star D \to W \star E$ is an equivalence in \f C.
\end{enumerate}
\end{Proposition}
\begin{proof}
To say that $q_{W} \star D:QW \star D \to W \star D$ is an equivalence is equally to say that the functor $\f C(q_{W} \star D,A) : \f C(W \star D,A) \to \f C(QW \star D,A)$ is an equivalence for each $A \in \f C$.  The canonical isomorphisms render this isomorphic to the inclusion $[\f J,\Cat](W, \f C(D-,A)) \to Ps(\f J,\Cat)(W, \f C(D-,A))$ so that (1) and (2) are equivalent.

That (1) implies (3) is also straightforward.  Assuming (1) consider a pointwise equivalence $f:D \to E$ of diagrams in $[\f J^{op},\f C]$.  This induces a pointwise equivalence $f^{*}:C(E-,A) \to C(D-,A) \in [\f J, \Cat]$.  Whiskering by $f^{*}$ induces a commutative square
$$\xy
(0,0)*+{[\f J,\Cat](W, \f C(E-,A))}="a0"; (50,0)*+{Ps(\f J,\Cat)(W, \f C(E-,A))}="b0";(0,-10)*+{[\f J,\Cat](W, \f C(D-,A))}="c0";(50,-10)*+{Ps(\f J,\Cat)(W, \f C(D-,A))}="d0";
{\ar^{} "a0"; "b0"}; 
{\ar_{} "a0"; "c0"}; 
{\ar^{} "b0"; "d0"}; 
{\ar_{} "c0"; "d0"}; 
\endxy$$
in which the horizontal arrows are the inclusions, both of which are equivalences since $W$-colimits are bicolimits in \f C.  The pointwise equivalence $f^{*}$ is a genuine equivalence in $Ps(\f J,\Cat)$ so that the right vertical arrow is an equivalence whence, by 2 out of 3, the left vertical arrow is an equivalence too.  Therefore the isomorphic $\f C(W \star f,A):\f C(W \star E, A) \to \f C(W \star D, A)$ is an equivalence for each $A$ so that $W \star f$ is itself an equivalence.

It remains to show that (3) implies (2).  Recall the adjunction $(\iota:[\f J, \f C] \leftrightarrows Ps(\f J, \f C):Q^{c})$ from 2.2 with counit $q^{c}:Q^{c} \to 1$.  The key point here is that we have an isomorphism $\lambda:QW \star D \cong W \star Q^{c}D$ compatible with the counits in the sense that the triangle
$$
\xy
(0,7)*+{QW \star D}="a"; (0,-7)*+{W \star Q^{c}D}="b"; (30,0)*+{W \star D}="c"; 
{\ar@<0ex>_{\lambda} "a"; "b"}; 
{\ar@<0ex>_{W \star q^{c}_{D}} "b"; "c"}; 
{\ar@<0ex>^{q_{W} \star D} "a"; "c"}; 
\endxy
$$
commutes.  This is just the colimit variant of a result on limits of \cite{Gambino2008Homotopy}.  To see where the isomorphism comes from recall that the cotensor $A^{X}$ (or power) of an object $A \in \f C$ by a category $X$ is the limit defined by a natural isomorphism $$ \f C(B,A^{X}) \cong \Cat(X, \f C(B,A))$$
Given a weight $W: \f J \to \Cat$ and diagram $D: \f J^{op} \to \f C$ this isomorphism extends in a pointwise manner to a natural isomorphism as on the left below $$[\f J^{op}, \f C](D,A^{W-}) \cong [\f J, \Cat](W, \f C(D-,A))$$  which, moreover, is compatible with pseudonatural transformations in the sense that $Ps(\f J^{op}, \f C)(D,A^{W-}) \cong Ps(\f J, \Cat)(W, \f C(D-,A))$.  Applying these isomorphisms back and forth, together with the universal properties of $QW$ and $Q^{c}D$, yields an isomorphism $[ \f J,\Cat](QW, \f C(D-,A)) \cong [\f J, \Cat](W, \f C(Q^{c}D-,A))$ and so the claimed $\lambda:QW \star D \cong W \star Q^{c}D$.  It is straightforward to check compatibility with the counit.

As discussed in 2.2 the map $q^{c}_{D}$ is always a pointwise equivalence so that, by assumption, $W \star q^{c}_{D}$ is an equivalence.  As $\lambda$ is an isomorphism it follows that $q_{W} \star D$ is an equivalence for each $D$ proving (2).
\end{proof}

\begin{Corollary}\label{cor:fcols}\textnormal{(Makkai-Par{\'{e}})}
Filtered colimits are bicolimits in \Cat.
\end{Corollary}
\begin{proof}
As \Cat is a (finitely) combinatorial model category a filtered colimit of equivalences is again an equivalence \cite{Dugger2001Combinatorial}.  The equivalence of (1) and (3) in Proposition~\ref{prop:bicolimits} gives the result.
\end{proof}

\subsection{The final decomposition}

Our terminology below differs slightly from that used by Rosick{\'{y}} in 3.2 and 3.3 of \cite{Rosicky2012Rigidification}.  Given a suitable monoidal model category $\f V$, such as our \Cat or \Sset, and a weight $W \in [\f J, \f V]$ with cofibrant replacement $q_{c}:W_{c} \to W$ in the projective model structure, he calls $W$ \emph{homotopy invariant} if for each objectwise cofibrant diagram $D:\f J^{op} \to \f V$ the induced morphism $q_{c} \star D:W_{c} \star D \to W \star D$ is a weak equivalence in \f V.  In \Cat the cofibrant replacement of a weight $W$ is the map $q_{W}:QW \to W$ considered throughout whilst every object in \Cat is cofibrant; it now follows from Proposition~\ref{prop:bicolimits} that a weight $W:\f J \to \Cat$ is homotopy invariant in Rosick{\'{y}}'s sense just when $W$-colimits are bicolimits in \Cat.

For the final decomposition we need a notion of closure which quantifies over diagrams as well as weights.  Consider a full subcategory $:\f A \to \f B$ and a class $\Phi$ of pairs of weights and diagrams $\Phi=\{(W_{i}:\f J_{i} \to \Cat, D_{i}:\f J_{i}^{op} \to \f B); i \in I\}$ with each colimit $W_{i} \star D_{i}$ existing in \f B.  By the \emph{1-step closure} of \f A in \f B under colimits of type $\Phi$ we mean the full subcategory of \f B consisting of \f A together with each colimit of the form $W_{i} \star D_{i}$ for $(W_{i},D_{i}) \in \Phi$ \emph{where $D_{i}$ is a diagram taking values in \f A}.  In a similar manner the \emph{2-step closure} is obtained by adding $\Phi$-colimits \emph{with diagrams taking values in the 1-step closure}, and so on.\begin{footnote}{The limit of this process is the \emph{closure} discussed in 3.5 of \cite{Kelly1982Basic}.}\end{footnote}

\begin{Theorem}\label{thm:Decomp2}
Each flexible weight of $[\f J,\Cat]$ lies in the (4-step) closure of the finite coproducts of representables in $[\f J,\Cat]$ under colimits of type $\Phi$, where a pair $(W:\f K \to \Cat, D:\f K^{op} \to [\f J,\Cat])$ belongs to $\Phi$ just when
\begin{enumerate}
\item $W$ is a sifted weight.
\item $W$-colimits are bicolimits in \Cat.
\item Each diagram $D$ takes values amongst flexible weights and each colimit $W \star D$ is flexible.
\end{enumerate}
\end{Theorem}
\begin{proof}
Arbitrary coproducts of representables are flexible by 2.4.  In any category, or 2-category, an arbitrary coproduct can be constructed using filtered colimits of finite coproducts---in particular each coproduct of representables can be constructed as a filtered colimit of finite coproducts of representables in $[\f J,\Cat]$.  Filtered colimits certainly commute with finite products in \Cat, and so are sifted, and are bicolimits by Corollary~\ref{cor:fcols}; thus arbitary coproducts lie in the 1-step closure.  

We saw in Theorem~\ref{thm:Decomp1} that any flexible weight can be constructed from coproducts of representables in three stages by taking sifted flexible colimits, always of flexible weights.  As discussed in 4.1 flexible colimits are bicolimits in $\Cat$---thus each flexible weight lies in the 4-step closure.
\end{proof}

\bibliographystyle{acm}
\bibliography{bibdata}

\begin{thebibliography}{10}

\bibitem{Albert1988The-closure}
{\sc Albert, M.~H., and Kelly, G.~M.}
\newblock The closure of a class of colimits.
\newblock {\em Journal of Pure and Applied Algebra 51}, 1-2 (1988), 1--17.

\bibitem{Badzioch2002Algebraic}
{\sc Badzioch, B.}
\newblock Algebraic theories in homotopy theory.
\newblock {\em Annals of Mathematics 155}, 3 (2002), 895--913.

\bibitem{Bergner2005Rigidification}
{\sc Bergner, J.~E.}
\newblock Rigidification of algebras over multi-sorted theories.
\newblock {\em Alg. Geom. Topology 6\/} (2005), 1925--1955.

\bibitem{Bird1989Flexible}
{\sc Bird, G.~J., Kelly, G.~M., Power, A.~J., and Street, R.}
\newblock Flexible limits for 2-categories.
\newblock {\em Journal of Pure and Applied Algebra 61}, 1 (1989), 1--27.

\bibitem{Blackwell1989Two-dimensional}
{\sc Blackwell, R., Kelly, G.~M., and Power, A.~J.}
\newblock Two-dimensional monad theory.
\newblock {\em Journal of Pure and Applied Algebra 59}, 1 (1989), 1--41.

\bibitem{Bourke2011On-semiflexible}
{\sc Bourke, J., and Garner, R.}
\newblock On semiflexible, flexible and pie algebras.
\newblock {\em Journal of Pure and Applied Algebra\/} (In press).

\bibitem{Dugger2001Combinatorial}
{\sc Dugger, D.}
\newblock Combinatorial model categories have presentations.
\newblock {\em Advances in Mathematics 164}, 1 (2001), 177--201.

\bibitem{Gambino2008Homotopy}
{\sc Gambino, N.}
\newblock Homotopy limits for 2-categories.
\newblock {\em Mathematical Proceedings of the Cambridge Philosophical Society
  145}, 1 (2008), 43--63.

\bibitem{Hovey1999Model}
{\sc Hovey, M.}
\newblock {\em Model categories}, vol.~63 of {\em Mathematical Surveys and
  Monographs}.
\newblock American Mathematical Society, 1999.

\bibitem{Kelly1974Doctrinal}
{\sc Kelly, G.~M.}
\newblock Doctrinal adjunction.
\newblock In {\em Category Seminar (Sydney, 1972/1973)}, vol.~420 of {\em
  Lecture Notes in Mathematics}. Springer, 1974, pp.~257--280.

\bibitem{Kelly1982Basic}
{\sc Kelly, G.~M.}
\newblock {\em Basic concepts of enriched category theory}, vol.~64 of {\em
  London Mathematical Society Lecture Note Series}.
\newblock Cambridge University Press, 1982.

\bibitem{Kelly1989Elementary}
{\sc Kelly, G.~M.}
\newblock Elementary observations on {$2$}-categorical limits.
\newblock {\em Bulletin of the Australian Mathematical Society 39}, 2 (1989),
  301--317.

\bibitem{Kelly1993Coinverters}
{\sc Kelly, G.~M., Lack, S., and Walters, R. F.~C.}
\newblock Coinverters and categories of fractions for categories with
  structure.
\newblock {\em Appl. Categ. Structures 1}, 1 (1993), 95--102.

\bibitem{Kelly2005Notes}
{\sc Kelly, G.~M., and Schmitt, V.}
\newblock Notes on enriched categories with colimits of some class.
\newblock {\em Theory and Applications of Categories 14\/} (2005), 399--423.

\bibitem{Lack2002Codescent}
{\sc Lack, S.}
\newblock Codescent objects and coherence.
\newblock {\em Journal of Pure and Applied Algebra 175}, 1-3 (2002), 223--241.

\bibitem{Lack2007Homotopy-Theoretic}
{\sc Lack, S.}
\newblock Homotopy-theoretic aspects of 2-monads.
\newblock {\em Journal of Homotopy and Related Structures 7}, 2 (2007),
  229--260.

\bibitem{Lack2010A-2-categories}
{\sc Lack, S.}
\newblock A 2-categories companion.
\newblock In {\em Towards higher categories}, vol.~152 of {\em IMA Vol. Math.
  Appl.} Springer, 2010, pp.~105--191.

\bibitem{Lack2011Enhanced}
{\sc Lack, S., and Shulman, M.}
\newblock Enhanced 2-categories and limits for lax morphisms.
\newblock {\em Advances in Mathematics 229}, 1 (2011), 294--356.

\bibitem{Makkai1989Accessible}
{\sc Makkai, M., and Par{\'e}, R.}
\newblock {\em Accessible categories: the foundations of categorical model
  theory}, vol.~104 of {\em Contemporary Mathematics}.
\newblock American Mathematical Society, 1989.

\bibitem{Power1991A-characterization}
{\sc Power, J., and Robinson, E.}
\newblock A characterization of pie limits.
\newblock {\em Mathematical Proceedings of the Cambridge Philosophical Society
  110}, 1 (1991), 33--47.

\bibitem{Riehl2011Algebraic}
{\sc Riehl, E.}
\newblock Algebraic model structures.
\newblock {\em New York Journal of Mathematics 17\/} (2011), 17 (2011) 173--231
  173--231.

\bibitem{Rosicky2012Rigidification}
{\sc Rosick{\'y}, J.}
\newblock Rigidification of algebras over essentially algebraic theories.
\newblock Preprint, \url{arXiv:1206.0422v1} (2012).

\bibitem{Street2004Categorical-and}
{\sc Street, R.}
\newblock Categorical and combinatorial aspects of descent theory.
\newblock {\em Appl. Categ. Structures 12}, 5-6 (2004), 537--576.

\end{thebibliography}
\end{document}